\documentclass[12pt]{amsart}

\usepackage{amsmath,amsthm,amscd,amsfonts,amssymb,epic,eepic}
\usepackage[colorlinks=true,linkcolor=blue,citecolor=blue]{hyperref}

\newtheorem{thm}{Theorem}[section]
\newtheorem{cor}[thm]{Corollary}

\newtheorem{lem}[thm]{Lemma}

\newtheorem{prop}[thm]{Proposition}
\theoremstyle{remark}
\newtheorem*{rem}{Remark}

\newcounter{remarkscounter}
\newenvironment{remarks}
{\medskip\noindent{\it
Remarks.}\begin{list}{{\rm(\arabic{remarkscounter})}
}{\usecounter{remarkscounter}

\setlength{\labelsep}{\fill} \setlength{\leftmargin}{0pt}
\setlength{\itemindent}{\fill}
\setlength{\labelwidth}{\fill}\setlength{\topsep}{0pt}
\setlength{\listparindent}{0pt}}} {\end{list}}

\numberwithin{equation}{section}

\newcommand{\A}{\mathbb{A}}

\newcommand{\GL}{\mathrm{GL}}
\newcommand{\SL}{\mathrm{SL}}
\newcommand{\ZZ}{\mathbb{Z}}
\newcommand{\PGL}{\textrm{PGL}}
\newcommand{\FF}{\mathbb{F}}

\newcommand{\Gal}{\mathrm{Gal}}

\newcommand{\QQ}{\mathbb{Q}}

\newcommand{\lto}{\longrightarrow}

\newcommand{\OO}{\mathcal{O}}
\newcommand{\CC}{\mathbb{C}}
\newcommand{\RR}{\mathbb{R}}
\newcommand{\GG}{\mathbb{G}}

\newcommand{\pp}{\mathfrak{p}}

\newcommand{\bb}{\mathfrak{b}}
\newcommand{\cc}{\mathfrak{c}}

\newcommand{\Sym}{\mathrm{Sym}}

\newcommand{\quash}[1]{}

\theoremstyle{definition}

\renewcommand{\bar}{\overline}

\numberwithin{equation}{subsection}

\renewcommand{\hat}{\widehat}

\newcommand{\et}{\mathrm{\textrm{\'{e}t}}}

\newcommand{\Res}{\mathrm{Res}}

\newcommand{\Aut}{\mathrm{Aut}}

\begin{document}
\title{Algebraic cycles and Tate classes\\ on Hilbert modular varieties}
\author{Jayce R. Getz}
\author{Heekyoung Hahn}
\address{Department of Mathematics\\
Duke University\\
Durham, NC 27708}
\email{jgetz@math.duke.edu}
\address{Department of Mathematics\\
Duke University\\
Durham, NC 27708}
\email{hahn@math.duke.edu}

\thanks{The authors are thankful for partial support provided by NSERC Discovery Grants. }
\begin{abstract}
Let $E/\QQ$ be a totally real number field that is Galois over $\QQ$, and let $\pi$ be a cuspidal, nondihedral automorphic representation of $\GL_2(\A_E)$ that is in the lowest weight discrete series at every real place of $E$.   The representation $\pi$ cuts out a ``motive'' $M_{\et}(\pi^{\infty})$ from the $\ell$-adic middle degree intersection cohomology of an appropriate Hilbert modular variety.  If $\ell$ is sufficiently large in a sense that depends on $\pi$ we compute the dimension of the space of Tate classes in $M_{\et}(\pi^{\infty})$.  Moreover if the space of Tate classes on this motive over all finite abelian extensions $k/E$ 
is at most of rank one as a Hecke module, we prove that the space of Tate classes in $M_{\et}(\pi^{\infty})$ is spanned by algebraic cycles.  
\end{abstract}

\maketitle

\section{Introduction and a statement of the result}\label{intro}

Let $E/\QQ$ be a totally real number field of absolute degree $d$:
\begin{equation}\label{d}
d:=[E:\QQ].
\end{equation}
For each compact open subgroup of 
$U \leq \mathrm{Res}_{E/\QQ}\GL_2(\A^{\infty})=\GL_2(\A_E^{\infty})$ one has a Shimura variety 
$$
Y^U:=\mathrm{Sh}(\mathrm{Res}_{E/\QQ}\GL_2,(\CC-\RR)^d)^U.
$$
Here as usual $\A$ (resp. $\A_E$) is the adeles of $\QQ$ (resp. $E$), $\Res$ is the Weil restriction of scalars, and $\A^{\infty}$ (resp. $\A_E^{\infty}$) denotes the finite adeles of $\QQ$ (resp. $E$).
If $U$ is a neat subgroup then $Y^U$ is a quasi-projective (non-compact) smooth scheme over $\QQ$.  We denote by $X^U$ the Baily-Borel compactification of $Y^U$.  It is a projective scheme over $\QQ$ with isolated singularities.  
Any cohomology group of 
$Y^U$ or $X^U$ with coefficients in a ring $R$ comes equipped with an action of the Hecke algebra 
\begin{align}
C_c^{\infty}(\GL_2(\A_E^{\infty})//U):=C_c^{\infty}(\GL_2(\A_E^{\infty})//U,R)
\end{align}
of $U$-biinvariant compactly supported smooth functions with coefficients in $R$; it acts via correspondences.   

 Let $\pi$ be a cuspidal automorphic representation of $\GL_2(\A_E)$, let $\ell$ be a prime, let $\bar{\QQ}_{\ell}$ be an algebraic closure of $\QQ_{\ell}$ and let $\iota:\CC \to \bar{\QQ}_{\ell}$ be an isomorphism.  One can 
use these Hecke operators to cut out a subrepresentation $M_{\et,\iota}^U(\pi^{\infty})$ of the representation 
$H^d_{\et}(Y^U \times \bar{\QQ},\bar{\QQ}_{\ell})$ of $\Gal_{\QQ}$ that is $\pi^{\infty}$-isotypic\footnote{Here and below $\pi= \pi_{\infty} \otimes \pi^{\infty}$ where $\pi_{\infty}$ (resp.~$\pi^{\infty}$) is an admissible representation of $\GL_2(E_{\infty})$ (resp.~$\GL_2(\A_E^{\infty})$).} under the Hecke algebra.  Here $\Gal_k:=\Gal(\bar{\QQ}/k)$ is the absolute Galois group of number field $k$ with respect to a fixed choice of algebraic closure $\bar{\QQ}$ of $\QQ$. For a precise definition of $M_{\et, \iota}^U(\pi^{\infty})$ we refer the reader to \S \ref{coho} below.  We use the symbol $M$ because we think of this Galois representation as the \'etale realization of 
a motive, although we will not prove that this is the case.

\begin{rem} If $\pi$ is cuspidal, then the $\pi$-isotypic component of $H^j_{\et}(Y^U \times \bar{\QQ},\bar{\QQ}_{\ell})$ is nonzero only if $j =d$ (see, for example, \cite{Harder} and \cite{GG}).  If $\pi$ is noncuspidal, the $\pi$-isotypic component is explained in terms of either K\"ahler forms or Eisenstein series (see \cite{Harder}), and to address such classes would take us too far afield.
\end{rem}

In this paper, following the tradition of \cite{HLR}, \cite{MR} and \cite{RamaHil} our goal is to investigate how much of $M_{\et, \iota}^U(\pi^{\infty})$ is explained by the fundamental classes of algebraic cycles.  Assume that $d$ is even.  A suitable extension of the Tate conjecture \cite{Tate} to this non-compact setting is that for each finite extension $k/\QQ$ the space $M_{\et, \iota}^U(\pi^{\infty})(d/2)^{\Gal_k}$ is spanned by the classes of algebraic cycles on $Y^U \times k$.  Here $(n)$ denotes the $n$-fold Tate twist. Harder, Langlands and Rapoport \cite{HLR} proved this when $d=2$ in the nondihedral case, Murty and Ramakrishnan \cite{MR} dealt with the dihedral\footnote{Recall that an automorphic representation $\pi$ of $\GL_2(\A_E)$ is \emph{dihedral} if there is a quadratic extension $L/E$ and a Hecke character $\chi:\A_L^{\times} \to \CC^{\times}$ such that $\pi=AI(\chi)$ is the automorphic induction of $\chi$.}  case when $d=2$, and Ramakrishnan \cite{RamaHil} provided some results when $d=4$.   

We prove the following modest extension of the results of \cite{HLR}:
\begin{thm}\label{thm-main}  Let $U \leq \GL_2(\A_E^{\infty})$ be a compact open subgroup and let $\pi$ be nondihedral. Suppose that the rank of $M_{\et,\iota}^{U, ss}(\pi^{\infty})(d/2)^{\Gal_{k}}$ as a $C_c^{\infty}(\GL_2(\A_E^{\infty})//U)$ module is at most $1$ for all abelian extensions $k/E$.  If $\ell$ is sufficiently large in a sense depending on $\pi$ then $M_{\et,\iota}^{U,ss}(\pi^{\infty})(d/2)^{\Gal_k}$ is spanned by algebraic cycles for all finite extensions $k/\QQ$.  
\end{thm}
Here $M_{\et,\iota}^{U, ss}(\pi^{\infty})$
denotes the semisimplification of $M_{\et,\iota}^U(\pi^{\infty})$ as a $\Gal_{\QQ}$-representation.

\begin{remarks} 
\item In the course of proving Theorem \ref{thm-main} we compute $M_{\et, \iota}^{U,ss}(\pi^{\infty})(d/2)^{\Gal_{k}}$ over sufficiently large finite extensions $k/\QQ$ if $\pi$ is not dihedral (see Proposition \ref{one}).  This proposition, and perhaps the main theorem, may be well-known to experts, but they have not appeared in the literature and seem a useful springboard to further investigations of Tate classes on Hilbert modular varieties and more general Shimura varieties.

\item In order to prove our theorem we need to assume that the algebraic envelope of the Galois representation $\rho_{\pi,\iota}$ attached to $\pi$ by Blasius and Rogawski \cite{BlR} and Taylor \cite{T} is large.  We can deduce this from the results of \cite{Dimi} provided that $\ell$ is sufficiently large, and this is the reason for our assumption that $\ell$ is sufficiently large.  If one could strengthen the results of \cite{Dimi} as in \cite{Ribet}, then one could dispense with this assumption.  Alternately, if one knew the automorphy (or perhaps even potential automorphy) of certain tensor product representations  (compare \S \ref{sec-tens-prod}) one might place $\rho_{\pi,\iota}$ in a compatible system and then try to prove that the dimension of the space of Tate classes is independent of $\ell$.  
\end{remarks}

We now outline the contents of this paper.  In the next section we compute the set of one dimensional subrepresentations of tensor products of the standard representation of the algebraic group $\GL_2$.  In \S \ref{sec-alg-env} and \S \ref{sec-tens-prod} we use this result together with some ideas of Serre and Ribet to compute the number of one dimensional subrepresentations of certain tensor products of rank two Galois representations.  We set notation for Asai representations in \S \ref{sec-Asai} and state how the \'etale cohomology of Hilbert modular varieties is described as a Galois representation in terms of these Asai representations in \S \ref{coho}.  We recall Hirzebruch-Zagier cycles and their twists in \S \ref{sec-twist} and \S \ref{sec-HZ}
and also recall how their cohomological nontriviality is linked to Asai $L$-functions.  Finally in \S \ref{sec-main} we put these pieces together to prove Theorem \ref{main}, which is  a restatement of Theorem \ref{thm-main}.

\section{Counting one dimensional constituents of representations of $\GL_2$} \label{sec-alg-reps}

Let $R:\GL_{2} \to \GL_{2}$ denote the standard representation, where we regard $\GL_2$ as an algebraic group over a characteristic zero field.
Each irreducible representation of $\GL_{2}$ is isomorphic to $\Sym^kR\otimes (\wedge^2R)^{\otimes m}$ for some nonnegative integer $k$ and some integer $m$.  This is the representation of highest weight $(k,m)$. In this representation the diagonal torus $\begin{pmatrix}a &0\\0& a^{-1}b\end{pmatrix}$ acts with weights 
\begin{equation}\label{km}
a^kb^m, a^{k-2}b^{1+m}, \dots, a^{-k}b^{k+m}.
\end{equation}

\begin{lem} \label{even}
For any $n \in \ZZ_{> 0}$ the number of one dimensional subrepresentations of $\otimes_{i=1}^{2n}R$ is equal to
\begin{equation}
\binom{2n}{n}-\binom{2n}{ n-1}.
\end{equation}
Each one dimensional subrepresentation is isomorphic to the irreducible representation of highest weight $(0,n)$, namely $\det(R)^{n}$.
\end{lem}

\begin{proof}

Consider the binomial expansion
\begin{align} 
(a+a^{-1}b)^{2n}
=&\binom{2n}{0}a^{2n}+\binom{2n}{1}a^{2n-2}b+\cdots+\binom{2n}{n-1}a^{2}b^{n-1}+\binom{2n}{n}b^n\nonumber\\
&+\binom{2n}{n+1}a^{-2}b^{n+1}+\cdots+\binom{2n}{2n}a^{-2n}b^{2n}.\label{remaining}
\end{align}
With a little thought, one sees that the weights that occur in $R^{\otimes 2n}$ are precisely the weights in this binomial expansion, each occurring with multiplicity equal to the given binomial coefficients.   
In particular, the highest weights that occur are 
\begin{equation}\label{list}
(2n, 0), (2n-2, 1), \dots, (2, n-1), (0,n)
\end{equation}
and, as recalled above, the diagonal torus acts with weights
\begin{equation}\label{ith}
a^{2n-2i}b^{i}, a^{2n-2(i+1)}b^{1+i}, \dots, b^n, \dots,  a^{-2n+2i}b^{2n+i}
\end{equation}
in each representation of highest weight $(2n-2i,i)$.  In particular, we see that any one dimensional subrepresentation is the representation of highest weight $(0,n)$, corresponding to $\det(R)^n$.  
We are left with computing how many times the highest weight $(0,n)$ occurs.  

First observe from \eqref{remaining} that the weight $b^n$ appears $\binom{2n}{n}$ times in total. Second observe from \eqref{ith} that it appears in each highest weight of the list \eqref{list}. Note that the weights appearing the second half of binomial expansion in \eqref{remaining} have already occurred in all the subrepresentations of highest weight not equal to $(0, n)$, namely $(2n, 0), (2n-2, 1), \dots, (2, n-1)$. Therefore by \eqref{ith} the number of copies of the weight $b^n$ appearing only in the highest weight $(0,n)$ is
$$
\binom{2n}{n}-\binom{2n}{n-1}.
$$
\end{proof}

\begin{lem}\label{odd}
For any $n \in \ZZ_{\geq 0}$ there are no one dimensional subrepresentations of $\otimes_{i=1}^{2n+1}R$.
\end{lem}

\begin{proof}
As in the proof of Lemma \ref{even}, the weights that occur in $R^{\otimes (2n+1)}$ are precisely the weights occurring in the binomial expansion 
\begin{align} 
(a+a^{-1}b)^{2n+1}
=&\binom{2n+1}{0}a^{2n+1}+\binom{2n+1}{1}a^{2n-1}b+\cdots+\binom{2n+1}{n}ab^{n}\nonumber\\
&+\binom{2n+1}{n+1}a^{-1}b^{n+1}+\cdots+\binom{2n+1}{2n+1}a^{-(2n+1)}b^{2n+1}.
\end{align}
Note that $R^{\otimes(2n+1)}$ does not have weights of the form $b^k$ for any $k$. This implies that $R^{\otimes(2n+1)}$ does not have one dimensional subrepresentations.
\end{proof}

\section{Algebraic envelopes of rank two Galois representations} \label{sec-alg-env}

Let $F$ be a number field, let $\ell$ be a rational prime and let $\bar{\QQ}_{\ell}$ be a choice of algebraic closure of $\QQ_{\ell}$.  Let $G$ be a $\bar{\QQ}_{\ell}$-algebraic group.
Recall that the \textbf{algebraic envelope} of an abstract group $A \leq G(\bar{\QQ}_{\ell})$ is the smallest algebraic subgroup of $G$ whose $\bar{\QQ}_{\ell}$-points contain $A$.  Similarly
the algebraic envelope of a representation $\rho:\Gal_F \to \GL_n(\bar{\QQ}_{\ell})$ is the algebraic envelope of $\rho(\Gal_F)$.  Here as above $\Gal_F$ is the absolute Galois group of $F$; all representations of this group will be assumed to be continuous.  

For a representation $\rho:\Gal_F \to \GL_n(\bar{\QQ}_{\ell})$ we denote by $\bar{\rho} : \Gal_F\to \GL_n(\bar{\FF}_{\ell})$ its reduction modulo $\ell$; it depends on the choice of a $\Gal_F$-stable $\bar{\ZZ}_{\ell}$-lattice (but only up to semi-simplification).  

\begin{lem} \label{lem-sl2} 
Suppose that $\bar{\rho}(\Gal_F)$ contains $\SL_2(\mathbb{F}_{\ell})$ for some $\ell\geq5$.  Then the algebraic envelope of $\rho$ contains $\SL_2$.
\end{lem}

\begin{proof}   Let $\ell\geq 5$ be such that $\bar{\rho}(\Gal_F)$ contains $\SL_2(\FF_{\ell})$.  It then follows from \cite[Theorem 2.1]{Ribet} that $\rho(\Gal_F)$ contains $\SL_2(\ZZ_{\ell})$.  Thus the Lie algebra of $\SL_2(\ZZ_{\ell})$ is contained in the Lie algebra of the $\bar{\QQ}_{\ell}$-points of the algebraic envelope of $\rho$
which implies by dimension considerations that the algebraic envelope of $\rho$ contains $\SL_2$.  \end{proof}

Let
\begin{align} \label{rho-i}
\rho_{ i}: \Gal_F \longrightarrow \GL_2(\bar\QQ_{\ell}),\quad 1\leq i\leq n,
\end{align}
be a set of Galois representations.  
One has the tensor product
\begin{align} \label{tens-prod}
\otimes_{i=1}^{n}\rho_i: \Gal_F \longrightarrow \Aut\big(\otimes_{i=1}^{n}V_{i}\big)(\bar{\QQ}_{\ell}),
\end{align}
where $V_{i}=\QQ^2$ and we view $\Aut\big(\otimes_{i=1}^{n}V_{ i}\big)$ as an algebraic group over $\QQ$ (or by base change as an algebraic group over $\bar{\QQ}_{\ell}$).

One has a natural homomorphism 
\begin{equation}\label{phi}
\phi: \Gal_F \longrightarrow \prod_{i=1}^{n}\Aut(V_i)(\bar{\QQ}_{\ell})
\end{equation}
such that  $\otimes_{i=1}^{n} \rho_{ i}$ is the composite
\begin{align}
\Gal_F \lto \prod_{i=1}^{n}\Aut(V_i)(\bar{\QQ}_{\ell})  \lto \Aut\big(\otimes_{i=1}^{n}V_i\big)(\bar{\QQ}_{\ell}),
\end{align}
where the first map is $\phi$ and the second is the tensor product.  Here we are viewing $\mathrm{Aut}(V_i)$ for each $i$ as an algebraic group over $\QQ$.

\begin{lem}\label{image_phi}
Suppose that for each $j$ the group $\bar{\rho}_{ j}(\Gal_F)$ contains $\SL_2(\FF_{\ell})$ and that for each $i \neq j$ there is no character $\chi:\Gal_F \to \GL_1(\bar{\QQ}_{\ell})$ such that
$$
\rho_{ i}\cong \rho_{ j} \otimes \chi.
$$
  Then the algebraic envelope of $\phi(\Gal_F)$ contains
$\prod_{i=1}^{n}\SL_2$.
\end{lem}

\begin{proof}
Let $H$ be the algebraic envelope of $\phi(\Gal_F)$ inside $\prod_{k=1}^{n}\mathrm{Aut}(V_k)$.  
For each $i \neq j$ let
\begin{align}\label{proj}
P_{i j} : \prod_{k=1}^{n}\Aut(V_k) \longrightarrow \Aut(V_i)\times \Aut(V_j) \nonumber
\end{align}
be the natural projections.  By Lemma \ref{lem-sl2} the projection of $P_{ij}(H)$ to either $\mathrm{Aut}(V_i)$ or $\mathrm{Aut}(V_j)$ contains $\SL_2$.  By Goursat's lemma (in the context of algebraic groups or Lie algebras, compare \cite[Lemma 3.2]{Ribet} and \cite[Exercise 1.4.8]{Bourbaki}) and our assumption on the $\rho_{ i}$ together with the fact that all (algebraic) automorphisms of $\SL_2$ are inner we conclude that the projection $P_{ij}(H)$ contains $\SL_2 \times \SL_2$.   Note that $\SL_2(\bar{\QQ}_{\ell})$ is equal to its commutator subgroup as an abstract group.  With this in mind, in spite of the fact that $\SL_2(\bar{\QQ}_{\ell})$ is infinite, 
 the proof of \cite[Lemma 3.3]{Ribet} is still valid in the present context and allows us to conclude the proof of the lemma.
 
 \end{proof}

\section{Tensor products of Galois conjugates} \label{sec-tens-prod}

Let $E/F$ be a Galois extension of number fields and let $\rho:\Gal_E \to \GL_2(\bar{\QQ}_{\ell})$ be a Galois representation.
We now investigate the number of one dimensional subrepresentations of the tensor product $\otimes_{\zeta\in\Gal(E/F)} \rho^{\zeta}$.

For fixed $\ell$, define an equivalence relation $\sim$  on the set of  $\ell$-adic Galois representations $\rho:\Gal_E \to \GL_2(\bar{\QQ}_{\ell})$ by $\rho_{ 1}\sim \rho_{ 2}$ if $P\rho_{ 1}\cong P\rho_{ 2}$ as projective representations, where $P\rho_{ i}$ is the composite
$$
\begin{CD}
P\rho_{ i}: \Gal_E@> {\rho_{i}}>>\GL_2(\bar{\QQ}_{\ell})@> {P}>>\PGL_2(\bar{\QQ}_{\ell})
\end{CD}
$$
with $P$ the natural projection.
Note that $\Gal (E/F)$ acts on the set of equivalence classes. For  each $\rho : \Gal_E \longrightarrow \GL_2(\bar\QQ_{\ell})$, let $\Gal(E/F)_{\rho}$ denote the stabilizer of the equivalence class of $\rho$.
Thus there are characters $\chi(\xi):\Gal_E \to \GL_1(\bar{\QQ}_{\ell})$ indexed by $\xi \in \Gal(E/F)_{\rho}$ such that
$$
\rho^{\xi} \cong \rho \otimes \chi(\xi).
$$
Thus
\begin{equation}\label{stable}
\bigotimes_{\xi \in \Gal(E/F)_{\rho}}\rho^{\xi} \cong  \rho^{\otimes |\Gal(E/F)_{\rho}|}\prod_{\xi \in \Gal(E/F)_{\rho}}\chi(\xi).
\end{equation}

\begin{prop}\label{one}
Suppose that the algebraic envelope of $\rho$ contains
$\SL_2$.
Let 
$$
m=|\Gal(E/F)_{\rho}|.
$$
Then for any finite extension $k/E$ the number of one dimensional subrepresentations of $$\otimes_{\zeta\in\Gal(E/F)}\rho^{\zeta}|_{\Gal_k}$$ is zero if 
$m$ is odd and otherwise is equal to 
$$
\left(\binom{m}{m/2}-\binom{m}{m/2-1}\right)^{[E:F]/m}.
$$
In the latter case each one dimensional subrepresentation is isomorphic to 
\begin{align*}
\prod_{\mu \in \Gal(E/F)/\Gal(E/F)_{\rho}}\left(\det(\rho)^{m/2}\prod_{\xi \in \Gal(E/F)_{\rho}}\chi(\xi)\right)^{\mu}\Bigg|_{\Gal_k}.
\end{align*}
\end{prop}

Before proving the proposition we state a corollary of this result and Lemma \ref{lem-sl2}:

\begin{cor}\label{one-K}
Suppose that $\rho:\Gal_E \to \GL_2(\bar{\QQ}_{\ell})$, $\ell \geq 5$, has the property that $\bar{\rho}(\Gal_E)$ contains $\mathrm{SL}_2(\FF_{\ell})$. If there is a one dimensional subrepresentation of $\otimes_{\zeta \in\Gal(E/F)} \rho^{\zeta}$ then there is a subfield $E \geq K \geq F$ with $[E:K]=2$, $\Gal(E/K)=\langle \sigma \rangle$ and a character $\chi: \Gal_E \to \GL_1(\bar{\QQ}_{\ell})$ such that $\rho^{\sigma} \cong \rho \otimes \chi$. \qed
\end{cor}

\begin{proof}[Proof of Proposition \ref{one}]
 
Since $\Gal(E/F)$ acts on the set of equivalence classes of $\rho$, using the notation of \eqref{stable} we write
{\allowdisplaybreaks
\begin{align*}
\otimes_{\zeta \in\Gal(E/F)} \rho^{\zeta}=&\otimes_{\mu \in \Gal(E/F)/\Gal(E/F)_{\rho}}\left(\otimes_{\xi \in \Gal(E/F)_{\rho}}\rho^{\xi}\right)^{\mu}\\
=&\otimes_{\mu \in \Gal(E/F)/\Gal(E/F)_{\rho}}\left(\rho^{\otimes |\Gal(E/F)_{\rho}|}\prod_{\xi \in \Gal(E/F)_{\rho}}\chi(\xi)\right)^{\mu}\\
=&\otimes_{\mu \in \Gal(E/F)/\Gal(E/F)_{\rho}}\left(\rho^{\otimes m}\prod_{\xi \in \Gal(E/F)_{\rho}}\chi(\xi)\right)^{\mu}.
\end{align*}}

\noindent
By Lemma \ref{image_phi} and this decomposition we see that
the number of one dimensional subrepresentations of $\otimes_{\zeta\in\Gal(E/F)}\rho^{\zeta}$ is equal to the number of one dimensional subrepresentations of $\rho^{\otimes m}$ taken to the $|\Gal(E/F)|/m$ power. Hence the result follows from Lemma \ref{even} and Lemma \ref{odd}.
\end{proof}


\section{Asai $L$-functions}

\label{sec-Asai}

Let $E/F$ be an extension of number fields. Fix $n\geq 1$. One has the \emph{Asai representation}
\begin{align}
\mathrm{As}_{E/F}:{}^L\mathrm{Res}_{E/F}\GL_n=\GL_n(\bar{\QQ}_{\ell})^{\mathrm{Hom}_{F}(E,\bar{F})} \rtimes \Gal_F \lto \GL((\bar{\QQ}_{\ell}^n)^{\otimes \mathrm{Hom}_{F}(E,\bar{F})})
\end{align}
defined by stipulating that
\begin{align*}
\mathrm{As}_{E/F}(((g_{\sigma})_{\sigma \in \mathrm{Hom}_{F}(E,\bar{F})},1))(\otimes_{\sigma \in \mathrm{Hom}_{F}(E,\bar{F})}v_{\sigma} )=\otimes_{\sigma \in \mathrm{Hom}_{F}(E,\bar{F})}g_{\sigma}v_{\sigma}\\
\mathrm{As}_{E/F}((1)_{\sigma \in \mathrm{Hom}_F(E,\bar{F})},\tau)(\otimes_{\sigma \in \mathrm{Hom}_{F}(E, \bar{F})}v_{\sigma})=\otimes_{\sigma \in \mathrm{Hom}_{F}(E,\bar{F})}v_{\tau \circ \sigma}.
\end{align*}

A representation 
$$
\rho:\Gal_E \lto \GL_n(\bar{\QQ}_{\ell})
$$
extends uniquely to a homomorphism
$$
\rho:\Gal_F \lto {}^L\mathrm{Res}_{E/F}\GL_n
$$
commuting with the projections to $\Gal_F$ on the $L$-group side.
Thus to each such $\rho$ we 
can associate the representation
$$
\mathrm{As}_{E/F}(\rho):=\mathrm{As}_{E/F}\circ\rho :\Gal_F \lto \GL((\bar{\QQ}_{\ell}^{n})^{\otimes \mathrm{Hom}_{F}(E,\bar{F})}).
$$

We note that for all field extensions $L \geq E$ one has
$$
\mathrm{As}_{E/F}(\rho)\big|_{\Gal_L} \cong \otimes_{\sigma \in \mathrm{Hom}_{F}(E,\bar{F})}\rho^{\sigma}\big|_{\Gal_L}.
$$
Thus the Asai representation $\mathrm{As}_{E/F}(\rho)$ is a canonical extension of $\otimes_{\sigma \in \mathrm{Hom}_{F}(E,\bar{F})}\rho^{\sigma}$ to $\Gal_{F}$.


\section{Certain cohomology groups}\label{coho}

For the basic results on Shimura varieties used without further comment in this section we refer to the reader to \cite{Deligne}. We view the pair $(\mathrm{Res}_{E/\QQ}\GL_2,(\CC-\RR)^d)$ as a Shimura datum in the usual manner \cite[\S 5.1]{GG}.  Thus we have, for each compact open subgroup 
$U \leq \mathrm{Res}_{E/\QQ}\GL_{2}(\A^{\infty})=\GL_2(\A_E^{\infty})$ a (finite level) Shimura variety
$$
Y^U=\mathrm{Sh}(\mathrm{Res}_{E/\QQ}\GL_2,(\CC-\RR)^d)^U.
$$
This is a quasi-projective scheme over $\QQ$.  If $U$ is neat, then $Y^U$ is smooth.  We denote by $X^U$ the Bailey-Borel compactification of $Y^U$; it is a projective scheme over $\QQ$ with isolated singularities.

Consider the cohomology groups
\begin{align}
M_{\et}^U:&=\mathrm{Im}\left(H^{d}_{\et,c}(Y^U \times\bar{\QQ},\bar{\QQ}_{\ell}) \lto H^{d}_{\et}(Y^U \times\bar{\QQ},\bar{\QQ}_{\ell})\right), \label{Met}\\ 
M_B^U:&=\mathrm{Im}\left(H^{d}_{B,c}(Y^U(\CC), \CC) \lto H^{d}_{B}(Y^U(\CC),\CC)\right),\label{MB}
\end{align} 
where the subscript $B$ denotes the Betti or singular cohomology.
These vector spaces over $\bar{\QQ}_{\ell}$ and $\CC$, respectively, are endowed with actions of the Hecke algebras
$$
C_c^{\infty}(\GL_2(\A_E^{\infty})//U,\bar{\QQ}_{\ell}) \quad\textrm{and}\quad C_c^{\infty}(\GL_2(\A^{\infty}_E)//U,\CC),
$$
respectively.  Upon choosing an isomorphism $\iota: \CC \to \bar{\QQ}_{\ell}$ one obtains canonical comparison isomorphisms
\begin{align} \label{comparison}
M_{B}^U \lto M_{\et}^U,
\end{align}
compatible with the action of Hecke operators.   We use the symbol $M^U$ because we think of the
objects above as motives, although we will not verify that they are motives in any rigorous sense.  To ease notation, we will henceforth omit the $\bar{\QQ}_{\ell}$ and $\CC$ from our notation for Hecke algebras; the coefficient ring will be clear from the context.

We note that $M_B^U$ can be viewed as an $L^2$-cohomology group, that is, 
$$
M_B^U \cong H_{(2)}^d(Y^U(\CC),\CC) \cong IH^d(X^U(\CC),\CC)
$$
as  a $C_c^{\infty}(\GL_2(\A_E^{\infty})//U)$-module, where $IH$ denotes intersection cohomology with middle perversity \cite[\S 7.2]{GG}.  
Moreover, 
\begin{align}
M_{\et}^U \cong IH_{\et}^d(X^U \times \bar{\QQ},\bar{\QQ}_{\ell})
\end{align}
as $\Gal_{\QQ} \times C_c^{\infty}(\GL_2(\A_E^{\infty})//U)$-modules (this follows from the result above and comparison isomorphisms in the context of \'etale intersection cohomology \cite[\S 6.1]{FP}).

For an admissible representation $\pi^{\infty}$ of $\GL_2(\A_{E}^{\infty})$ denote by $M^U_{B}(\pi^{\infty}) \leq M^U_B$ the $\pi^{\infty}$-isotypic component under the Hecke algebra.  Note that there is a decomposition
$$
M_B^U=\bigoplus_{\pi}M_B^U(\pi^{\infty})
$$
where the sum is over all automorphic representations $\pi$ of $\GL_2(\A_E)$ such that  $H^{d}(\mathfrak{g}, U_{\infty};  \pi_{\infty})\neq 0$ and $\pi$ is either cuspidal or the determinant mapping followed by a character \cite[\S 7.2]{GG}.  Here as usual, $\mathfrak{g}$ is the complexification of the Lie algebra of $\GL_2(E_{\infty})$ and $U_{\infty}=(\RR_{>0} \mathrm{SO}_2(\RR))^d$
 (to see this use \cite{BoCa}).
Applying the comparison isomorphisms in \'etale cohomology, the choice of $\iota$ induces an isomorphism
\begin{align}
M_B^U(\pi^{\infty}) \lto M_{\et,\iota}^U(\pi^{\infty}):=\iota(M_B^U(\pi^{\infty})).
\end{align}
If $\pi$ is a cuspidal automorphic representation of $\GL_2(\A_E)$ with $H^{d}(\mathfrak{g}, U_{\infty};  \pi_{\infty})\neq 0$
denote by 
\begin{align}
\rho_{\pi,\iota}:\Gal_E \lto \GL_2(\bar{\QQ}_{\ell})
\end{align}
the associated Galois representation \cite{BlR, T}.  It has the property that $\det(\rho_{\pi,\iota})$ is the cyclotomic character times a finite order character.
The following theorem \cite{BrL} gives a description of $M^U_{\et,\iota}(\pi^{\infty})$:

\begin{thm}[Brylinski-Labesse]
As a $\Gal_{\QQ}  \times C_c^{\infty}(\GL_2(\A_E^{\infty})//U)$-module one has that the semisimplifications of the representations
$$
M_{\et,\iota}^U(\pi^{\infty})\quad \textrm{and} \quad\mathrm{As}_{E/\QQ}(\rho_{\pi,\iota}) \otimes \pi^{\infty U}
$$
are isomorphic.
\end{thm}
Here the superscript $U$ means the vectors fixed by $U$. Thus $M_{\et,\iota}^{U}(\pi^{\infty})$ is of rank $2^{[E:\QQ]}$ as a $C_c^{\infty}(\GL_2(\A_E^{\infty})//U)$-module.   We also remark that by ``semisimplification'' we mean semisimplification as a $\Gal_{\QQ}$-module.  As in the introduction, we use a superscript ``${}^{ss}$'' to denote the semisimplification of a representation of $\Gal_{k}$ for number fields $k/\QQ$.

\begin{prop}\label{one-motive}
Assume that $\pi$ is not a dihedral representation, that $\ell$ is sufficiently large (in a sense depending on $\pi$), and
let $m=|\Gal(E/\QQ)_{\rho_{\pi,\iota}}|$. For even $d$,
there is an abelian extension $k$ of $E$ with the property that for all finite extensions $k'/k$ one has that
\begin{align}
M_{\et,\iota}^{U,ss}(\pi^{\infty})(d/2)^{\Gal_{k'}}
\end{align}
has rank $0$ as a $C_c^{\infty}(\GL_2(\A_E^{\infty})//U)$-module if $m$ is odd and otherwise has rank
\begin{align} \label{dim23}
\left(\binom{m}{m/2}-\binom{m}{m/2-1}\right)^{d/m}.
\end{align}
\end{prop}
\noindent In other words, if $m$ is even, \eqref{dim23} is the rank of the space of Tate cycles as a $C_c^{\infty}(\GL_2(\A_E^{\infty})//U)$-module.

\begin{proof}  
By \cite[Proposition 3.8]{Dimi} for all but finitely many $\ell$ the reduction of the image of $\rho_{\pi,\iota}$ modulo $\ell$ 
contains $\mathrm{SL}_2(\FF_{\ell})$.
Thus the proposition follows from 
Proposition \ref{one} together with the fact that $\det(\rho_{\pi,\iota})$ is the cyclotomic character times a finite order character. 
\end{proof}

\section{Twisting by characters}\label{sec-twist}

As explained by Murty and Ramakrishnan \cite{MR}, there is another natural family of correspondences on $M_?^U$, $?\in \{\et, B\}$, besides the Hecke correspondences, namely the twisting correspondences.  In this section we recall their construction following the exposition of \cite[\S 9.3]{GG}.

Fix an ideal $\cc\subset \OO_E$ and denote the congruence subgroup of integral matrices of ``Hecke type" by
$$
U_{1}(\cc):=\left\{\begin{pmatrix}a &b\\c&d\end{pmatrix}\in\GL_2(\hat{\OO}_E) \,:\,  d-1, \, c\in\cc\hat{\OO}_E \right\}.
$$Let $Y_{1}(\cc)$ denote the corresponding Hilbert modular variety, that is
$$
Y_{1}(\cc):=Y^{U_{1}(\cc)}.
$$Similarly we use $X_{1}(\cc)$ for its Baily-Borel compactification.
Let $\theta : E^{\times}\backslash \A_E^{\times}\lto\CC^{\times}$ be a finite order Hecke character with conductor $\bb$. 
To ease notation, we fix an isomorphism $\iota:\CC \to \bar{\QQ}_{\ell}$ and sometimes view $\theta$ as taking values in $\bar{\QQ}_{\ell}$; in other words we sometimes identify $\theta$ and $\iota \circ \theta$.
Suppose that $\bb\neq \OO_E$. Let $b$ be a finite id\`{e}le with $[b]=\bb$. Define the fractional ideal $\Upsilon=b^{-1}\hat{\OO}_E$ of $\prod_{\pp_v|\bb}E_v\times\prod_{\pp_v\nmid\bb}\OO_v$ by
$$
\Upsilon:=\left\{t=(t_v)\in\prod_{\pp_v|\bb}E_v\times\prod_{\pp_v\nmid\bb}\OO_v: \mathrm{ord}_v(t_v)\geq-\mathrm{ord}_v(b) \text{ whenever } \pp_v|\bb\right\}.
$$
Let $\widetilde{\Upsilon}=b^{-1}\hat{\OO}_E/\hat{\OO}_E$ be a set of representatives for $\Upsilon$ modulo $\hat{\OO}_E=\prod_v\OO_v$. Denote by $\theta_{\bb}:\widetilde{\Upsilon}\lto\CC$ the map defined by setting
$$
\theta_{\bb}(t)=\begin{cases} \theta(t)\quad \text{if } t\in \widetilde{\Upsilon}^{\times}:=\{\frac{x}{b} \in \widetilde{\Upsilon}: x \in \OO_E^{\times}\}\\0\quad\text{otherwise.}\end{cases}
$$
For $t\in \Upsilon$ define $u_t=u(t)\in \GL_2(\A_E)$ by $u(t)_v=\big(\begin{smallmatrix}1&0\\0&1\end{smallmatrix}\big)$ if $v\nmid \bb$ and $u(t)_v=\big(\begin{smallmatrix}1&t_v\\0&1\end{smallmatrix}\big)$ if $v|\bb$. Then there is a correspondence
\begin{equation}\label{corres}
\begin{CD}
Y_{1}(\cc\bb^2) @>{\cdot u_t}>> Y_{1}(\cc\bb^2) @> \pi >> Y_{1}(\cc),
\end{CD}
\end{equation}
where the second map $\pi$ is the projection map.  If $\bb=\OO_E$ then we replace this correspondence by the identity map $Y_{1}(\cc) \to Y_{1}(\cc)$.
Let $\CC_X$ be the locally constant sheaf with stalk $\CC$ on $X$ for topological spaces $X$.  If $X$ is a quasi-projective scheme over $\QQ$,  $w$ is a place above a rational prime $\ell$, $\varpi_w$ is a uniformizer of $\OO_w$, and $n$ is an integer, let ${(\OO_w/{\varpi_w^n})}_X$ be the locally constant sheaf with stalk $\OO_w/\varpi_w^n$ on the \'etale site of $X \times \bar{\QQ}$.  Assume that $w$ is chosen so that $\theta$ has values in $\OO_w$.  Then we have the following lemma (see \cite[Lemma 9.3]{GG}):

\begin{lem}\label{betty} For $\textbf{R} \in \{\CC,\OO_w/\varpi_w^n\}$ and
 $t\in\Upsilon$, the mapping (which we denote by $\cdot u_t$)
\begin{equation}\label{ut}
[g, v]\lto [g u_t, \theta (\det g)\theta_{\bb}(t)^{-1}v]
\end{equation}
gives a well-defined, canonical isomorphism
$$
\textbf{R}_{Y_{1}(\cc\bb^2)}\lto \textbf{R}_{Y_{1}(\cc\bb^2)}.
$$
Equivalently, the mapping $P=\pi\circ (\cdot u_t)$ defines a mapping
$$
P : \textbf{R}_{Y_{1}(\cc\bb^2)}\lto \textbf{R}_{Y_{1}(\cc)}.
$$ \qed
\end{lem}
Here when $\bb=\OO_E$ we set $\theta_{\bb}(t)=1$, let $u_t$ be the identity matrix and we replace $P$ by the identity mapping from $Y_{1}(\cc)$ to itself.  

It may not be evident that this definition makes sense in the \'etale setting because it appears to depend on $\det(x_{\infty})$.  To explain this, recall that
for all $U \leq \mathrm{Res}_{E/\QQ}\GL_2(\A^{\infty})$ the determinant $\det:\mathrm{Res}_{E/\QQ}\GL_2 \to \mathrm{Res}_{E/\QQ}\GG_m$ induces a morphism of Shimura varieties 
$$
Y^U \lto \pi_0(Y^U)
$$
where $\pi_0(Y^U)$ is the scheme of connected components of $Y^U$.
Letting $E_+^{\times}$ denote the group of totally positive elements of $E^{\times}$,  the definition of the lift of $P$ in fact only depends on the image of $\det(x)$ in 
$$
E^{\times} \RR^{[E:\QQ]}_{>0} \backslash \A_E^{\times}/\det(U_{1}(\cc\bb^2))=E_+^{\times} \backslash \A_E^{\infty \times}/\det(U_{1}(\cc\bb^2))=\pi_0(Y_{1}(\cc\bb^2))(\CC).
$$
In other words, the coefficients $\theta(x)$ depend only on the connected component that $x$ lies in, and these are all defined over a suitable finite extension of $\QQ$. 

Given a finite order Hecke character $\theta$ with conductor $\bb$, define the {\it twisting correspondence} $\mathcal{T}_{\theta}(\cc\bb^2)_*$ to be the disjoint union over $t\in \widetilde{\Upsilon}$ of the correspondences \eqref{corres}.  In the case $\textbf{R}=\CC$ this yields a map
\begin{align}
\mathcal{T}_{\theta}(\cc\bb^2)_*:M_B^{U_{1}(\cc\bb^2)} \lto M_B^{U_{1}(\cc)}.
\end{align}
In the \'etale case after passing to the limit over $n$ and tensoring with $\bar{k}_w \cong \bar{\QQ}_{\ell}$, where $k_w$ is the fraction field of $\OO_{w}$, we obtain
\begin{align}
\mathcal{T}_{\theta}(\cc\bb^2)_*:M_{\et}^{U_{1}(\cc\bb^2)} \lto M_{\et}^{U_{1}(\cc)}.
\end{align}
These correspondences are compatible with the comparison isomorphisms \eqref{comparison}.

As noted above, the component set is equipped with a canonical isomorphism $\pi_0(Y^U)(\CC) = E^{\times}_+ \backslash \A_E^{\infty \times}/\det(U)$.  We therefore have a composite map
$$
\theta:  Y_{1}(\cc\bb^2)(\CC) \lto \pi_0(Y_{1}(\cc\bb^2))(\CC) \lto \CC.
$$
The field of definition $E_{\theta}$ of $\theta$ is therefore defined.  Using the description of the $\Gal(\CC/\QQ)$ action on $\pi_0(Y_{1}(\cc\bb^2))$ (see \cite[\S 13, p. 349]{Mi} for details) we see that it is an abelian extension of $E$ contained in the narrow ring class field of conductor $\cc\bb^2$.  Using this fact, one checks the following lemma:

\begin{lem} 
The map $\mathcal{T}_{\theta}(\cc\bb^2)$ may be viewed as a finite $\bar{\QQ}_{\ell}$-linear combination of correspondences from $Y_{1}(\cc\bb^2) \times E_{\theta}$ to $Y_{1}(\cc) \times E_{\theta}$ that extend to correspondences on the Baily-Borel compactification. Thus $\mathcal{T}_{\theta}(\cc\bb^2)$ descends to a homomorphism
$$
\mathcal{T}_{\theta}(\cc\bb^2)_*:(M_{\et}^{U_{1}(\cc\bb^2)})^{\Gal_{E_{\theta}}} \lto (M_{\et}^{U_{1}(\cc)})^{\Gal_{E_{\theta}}}.
$$ \qed
\end{lem}

We note in addition that for each automorphic representation $\pi$ of $\GL_2(\A_E)$ the twisting correspondence induces push-forward and pull-back maps
\begin{align*}
\mathcal{T}_{\theta}(\cc \bb^2)_*:
M_B^{U_{1}(\cc\bb^2)}(\pi^{\infty} \otimes \theta^{\infty}) & \lto M_B^{U_{1}(\cc)}(\pi^{\infty}) \\
\mathcal{T}_{\theta}(\cc \bb^2)^*:M_B^{U_{1}(\cc)}(\pi^{\infty}) & \lto M_B^{U_{1}(\cc\bb^2)}(\pi^{\infty} \otimes \theta^{\infty})
\end{align*}
and hence
\begin{align*}
\mathcal{T}_{\theta}(\cc \bb^2)_*:M_{\et,\iota}^{U_{1}(\cc\bb^2)}(\pi^{\infty} \otimes \theta^{\infty})^{\Gal_{E_{\theta}}} &\lto M_{\et,\iota}^{U_{1}(\cc)}(\pi^{\infty})^{\Gal_{E_{\theta}}}\\
\mathcal{T}_{\theta}(\cc \bb^2)^*:M_{\et,\iota}^{U_{1}(\cc)}(\pi^{\infty} )^{\Gal_{E_{\theta}}} &\lto M_{\et,\iota}^{U_{1}
(\cc\bb^2)}(\pi^{\infty} \otimes \theta^{\infty})^{\Gal_{E_{\theta}}}
\end{align*}
(compare \cite[Proposition 9.4]{GG}).

\section{Twisted Hirzebruch-Zagier cycles} \label{sec-HZ}

Our goal in this section is to define twisted Hirzebruch-Zagier cycles and give a criterion  for their nontriviality.  

Suppose that $K \leq E$ is a quadratic subfield.  Using the definitions from above we obtain a Shimura variety
$$
Y_K^{U_K}
$$
for all compact open subgroups $U_K \leq \mathrm{Res}_{K/\QQ}\GL_2(\A^{\infty})$; it is again a quasi-projective scheme
over $\QQ$.  Fix $U \leq \mathrm{Res}_{E/\QQ}\GL_2(\A^{\infty})$, set $U_K:=U \cap \mathrm{Res}_{K/\QQ}\GL_2(\A^{\infty})$,
and consider the natural inclusion morphism
\begin{equation}\label{iota}
Y_K^{U_K} \lto Y^U.
\end{equation}
It induces Gysin maps
\begin{align} 
H_{\et,c}^{0}(Y_K^{U_K} \times\bar{\QQ},\,\bar{\QQ}_{\ell}) &\lto 
H_{\et,c}^{d}(Y^U \times \bar{\QQ},\,\bar{\QQ}_{\ell}(d/2)),\label{Gys-et}\\
H_{B,c}^{0}(Y_K^{U_K} (\CC), \,\CC) &\lto 
H_{B,c}^{d}(Y^U (\CC),\,\CC(d/2)).\label{Gys-B}
\end{align}
The images admit natural maps to $M_{\et}^U$ and  $M_{B}^U$ respectively.  

We will denote the image of a fundamental class under the composition of \eqref{Gys-et} (resp. \eqref{Gys-B}) and the maps to $M_{?}^U(d/2)$ by
\begin{equation}\label{HZcycle}
[Z^K] \in M_?^U(d/2)
\end{equation}
and refer to it as a {\it Hirzebruch-Zagier cycle}. 

Now for $\cc, \bb \subset\OO_E$, let $U=U_{1}(\cc \bb^2)$ and let $\theta:E^{\times} \backslash \A_E^{\times} \to \CC^{\times}$ be a finite order Hecke character of conductor $\bb$.  We let 
\begin{equation}\label{tHZcycle}
[Z^{K,\theta}]=\mathcal{T}_{\theta}(\cc\bb^2)_*[Z^K] \in M^{U_{1}(\cc)}_?(d/2)
\end{equation}
and refer to it as a {\it twisted Hirzebruch-Zagier cycle}.  

Let $\pi$ be a cuspidal automorphic representation of $\GL_2(\A_E)$ satisfying $H^{d}(\mathfrak{g},U_{\infty};\pi_{\infty}) \neq 0$ and let $\cc$ be the conductor of $\pi$ (for unexplained notation, see \S \ref{coho}). Let $K \leq E$ be a quadratic subfield and let $\theta:E^{\times} \backslash \A_E^{\times} \to \CC^{\times}$ be a finite order Hecke character of conductor $\bb$. We now state our basic criterion for the projection of $[Z^{K,\theta}]$ to  $M_{\et, \iota}^{U_{1}(\cc)}(\pi^{\infty})(d/2)$ to be nontrivial. 

\begin{thm}\label{nontrivial}
Suppose that $\pi$ is a base change of a cuspidal automorphic representation $\pi_K$ of $\GL_2(\A_K)$ with central character $\omega_K$ and that $\theta|_{\A_{K}^{\times}}=\omega_K \eta$, where $\langle \eta \rangle =\Gal(E/K)^{\wedge}$.
Then $[Z^{K,\theta}]$ projects nontrivially onto $M_{\et,\iota}^{U_{1}(\cc)}(\pi^{\infty})(d/2)$. 
\end{thm}

\noindent The statement that $\langle \eta \rangle=\Gal(E/K)^{\wedge}$ is simply the statement that $\eta$ is the Hecke character of $K$ corresponding to $E/K$ by class field theory.

\begin{proof}
In view of the compatibility of the construction of  $[Z^{K,\theta}]$ with the comparison isomorphism $M_B^{U_{1}(\cc)}(\pi^{\infty}) \cong M_{\et,\iota}^{U_{1}(\cc)}(\pi^{\infty})$ it suffices to show that under the given assumption $[Z^{K,\theta}]$ is nontrivial in $M_B^{U_{1}(\cc)}(\pi^{\infty})(d/2)$.  There is a canonical Hecke-equivariant isomorphism
$$
M_B^{U_{1}(\cc)}(d/2) \cong IH^d(X_{1}(\cc)(\CC),\CC(d/2))
$$
(see the remark below \cite[Theorem 7.1]{GG}).  Therefore it suffices to show that the canonical class attached to $[Z^{K,\theta}]$ in the intersection homology group $IH_d(X_{1}(\cc)(\CC),\CC(d/2))$ constructed in \cite[Chapter 9]{GG} pairs nontrivially with the $\pi^{\infty}$-isotypic component of $IH^d(X_{1}(\cc)(\CC),\CC(d/2))$ under the given assumption. 
By \cite[Theorem 10.1]{GG} the canonical class does pair nontrivially with the $\pi^{\infty}$-isotypic component and
this completes the proof of the theorem.  
\end{proof}


\section{Statement and proof of main theorem} \label{sec-main}

Let $E/\QQ$ be a totally real number field with \textit{even} degree $d$.  Let $\pi$ be a cuspidal automorphic representation of $\GL_2(\A_E)$ with
$H^d(\mathfrak{g},U_{\infty}; \pi^{\infty}) \neq 0$, and let $\rho_{\pi, \iota}:\Gal_E \to \GL_2(\bar{\QQ}_{\ell})$ be the $\ell$-adic Galois representation attached to $\pi$ and the isomorphism $\iota:\CC \to \bar{\QQ}_{\ell}$.
We are interested in providing algebraic cycles to account for classes in 
$$
M_{\et,\iota}^{U,ss}(\pi^{\infty})(d/2)^{\Gal_k}
$$
where $k/\QQ$ is a finite extension and $U \leq \GL_2(\A_E^{\infty})$ is a compact open subgroup.  The Tate conjecture (suitably generalized to the non-projective case) implies that 
this space should be spanned by (classes) of algebraic cycles over $k$.   
If $M_{\et,\iota}^{U, ss}(\pi^{\infty})(d/2)^{\Gal_k} =0$, then it is trivially spanned by algebraic cycles.  We deal with the next simplest case:


\begin{thm}\label{main}
 Let $U \leq \GL_2(\A_E^{\infty})$ be a compact open subgroup and let $\pi$ be nondihedral. Suppose that the rank of $M_{\et,\iota}^{U, ss}(\pi^{\infty})(d/2)^{\Gal_{k}}$ as a $C_c^{\infty}(\GL_2(\A_E^{\infty})//U)$ module is at most $1$ for all abelian extensions $k/E$.  If $\ell$ is sufficiently large in a sense depending on $\pi$ then $M_{\et,\iota}^{U,ss}(\pi^{\infty})(d/2)^{\Gal_k}$ is spanned by algebraic cycles for all finite extensions $k/\QQ$. 
\end{thm}

We will prove the theorem by showing that the only way for the space of Tate classes to be nonzero is if it is spanned by a Hirzebruch-Zagier cycle or one of its twists.

\begin{proof}[Proof of Theorem \ref{main}] 

Suppose that as a $C_{c}^{\infty}(\GL_2(\A_E^{\infty})//U)$ module, $M_{\et,\iota}^{U,ss}(\pi^{\infty})(d/2)^{\Gal_k}$ is of rank at most $1$ for all abelian extensions $k/E$.  Then, assuming $\ell$ is sufficiently large, from Proposition \ref{one-motive} we have that it is of rank $0$ for all finite extensions $k/E$ or it is of rank $1$ for all sufficiently large finite extensions $k/E$.  In the former case we have nothing to prove, so assume that it is of rank $1$ as a $C_c^{\infty}(\GL_2(\A_E^{\infty})//U)$-module for some $k/E$.  From Proposition \ref{one-motive} 
we then have that 
 $\Gal(E/\QQ)_{\rho_{\pi,\iota}}$ has order $2$.

Let $\QQ\leq K\leq E$ be the subfield of $E$ fixed by $\Gal(E/\QQ)_{\rho_{\pi,\iota}}$ and let $\Gal(E/K)=\langle\sigma\rangle$.  Choose a character $\chi: \Gal_E \to \GL_1(\bar{\QQ}_{\ell})$ such that $\rho_{\pi,\iota}^{\sigma} \cong \rho_{\pi ,\iota} \otimes \chi$.   Applying \cite[Theorem 2]{LapRog} and our assumption that $\pi$ is not dihedral we can and do write $\chi=\mu\mu^{-\sigma}$ for some character $\mu$. Therefore 
$$
\rho_{\pi,\iota}\otimes (\mu\mu^{-\sigma})\cong\rho_{\pi,\iota}^{\sigma}
$$which implies in turn that 
$$
\rho_{\pi,\iota}\otimes \mu\cong (\rho_{\pi,\iota}\otimes\mu)^{\sigma}.
$$
This implies, by cyclic base change \cite[\S 2]{Langlands}, that there exists a cuspidal automorphic representation $\pi_0$ of $\GL_2(\A_K)$ such that
$$\pi \otimes \mu\cong\pi_{0 E},$$
where $\pi_{0 E}$ is the base change of $\pi_0$ to $E$. Choose a  finite order character $\theta$ of $\A_E^{\times}$ such that $\theta|_{\A_K^{\times}}=\omega_{\pi_0}\eta$, where $\omega_{\pi_0}$ is the central character of $\pi_0$ and $\eta$ is the character attached to $E/K$ by class field theory (note that such a character exists by \cite[Lemma 2.1]{Hida}).  Let $\cc_0$ denote the conductor of $\pi_{0 E}$.
Then $[Z^{K,\theta}]$ projects nontrivially to $M_{\et,\iota}^{U_{1}(\cc_0)}(\pi_{0 E}^{\infty})(d/2)$ by Theorem \ref{nontrivial}. 
 By newform theory, the fact that $\pi_{0E}$ is a twist of $\pi$, and our assumptions, we have that
 $M_{\et,\iota}^{U_{1}(\cc_0),ss}(\pi_{0E}^{\infty})(d/2)^{\Gal_k}$ is of rank $1$ for all sufficiently large number fields $k$ and any element of it is a generator for $M_{\et,\iota}^{U, ss}(\pi_0^{\infty})(d/2)^{\Gal_k}$ as a $C_c^{\infty}(\GL_2(\A_E^{\infty})//U)$-module for any compact open subgroup $U \leq \GL_2(\A_E^{\infty})$.  

Now let $\cc$ be the conductor of $\pi$, let $k$ be a sufficiently large number field, and let $\phi \in M_{\et,\iota}^{U_{1}(\cc),ss}(\pi^{\infty})(d/2)^{\Gal_k}$ be a generator for this one-dimensional space.  If $\bb$ denotes the conductor of $\mu$, we have that
\begin{align}
0 \neq \mathcal{T}_{\mu}(\cc \bb^2)^*\phi \in M_{\et,\iota}^{U_{1}(\cc \bb^2),ss}(\pi_{0 E}^{\infty})(d/2)^{\Gal_k}
\end{align}
Thus
$$
\mathcal{T}_{\mu}(\cc \bb^2)_*\mathcal{T}_{\mu}(\cc \bb^2)^*\phi,
$$
which is a scalar multiple of $\phi$, is the class of an algebraic cycle.  Applying newform theory again we see that 
 $M_{\et,\iota}^{U,ss}(\pi^{\infty})(d/2)^{\Gal_k}$ is spanned by algebraic cycles for all sufficiently large number fields $k$ and all compact open subgroups $U \leq \GL_2(\A_E^{\infty})$.
 Since there is a compatible action of $\Gal_{\bar{\QQ}}$ on the space of algebraic cycles and the space of Tate classes, the theorem follows.

\end{proof}

\section*{Acknowledgement}

The authors thank the referee for useful comments on this paper.


\end{document}